\documentclass[11pt]{gtart}
\usepackage{amsmath, amssymb, graphicx, epsfig, verbatim}
\usepackage{epic}

\newtheorem{thm}{Theorem}[section]

\newtheorem{lemma}[thm]{Lemma}

\theoremstyle{definition}
\newtheorem{defn}[thm]{Definition}
\newtheorem{conj}[thm]{Conjecture}

\newtheorem{question}[thm]{Question}
\numberwithin{equation}{section}

\newcommand{\pp}{{\mathbf {t}}}

\newcommand{\D}{\mathbf d}
\newcommand{\DD}{\mathcal D}

\newcommand{\dd}{\mathcal D}

\newcommand{\Z}{\mathbb Z}

\newcommand{\cf}{{\mathcal  C}}
\newcommand{\ff}{{F}}

\newcommand{\rr}{\mathcal R}
\newcommand{\cff}{{\overline {\mathcal C}}}

\begin{document}

\title{ A geometric spectral sequence in Khovanov homology}

\author{Zolt\'an Szab\'o}

\address{Department of Mathematics\\
Princeton University,\\
 Princeton, NJ, 08544}

\email{szabo@math.princeton.edu}

\begin{abstract}
The aim of this paper is to introduce and study a geometric spectral sequence
on $\Z_2$ Khovanov homology.
\end{abstract}

\maketitle

\section{Introduction}

The construction in the present paper was motivated by a
joint work with Peter Ozsv\'ath on the Heegaard Floer homology of double branched covers. 
In
 \cite{OSdouble} 
 a spectral 
sequence is constructed  from  the reduced Khovanov homology of a link $L$
 to the hat version of 
 Heegaard Floer homology ${\widehat {HF}}(-\Sigma _L)$  with $\Z _2$
coefficients, where
  $\Sigma _L$  denotes the double cover of $S^3$ branched along $L$. 
The same construction also gives
a spectral sequence from the mod 2 Khovanov homology $Kh(L)$ to 
${\widehat {HF}}(-\Sigma _L\# (S^1\times S^2))$.
(Note that the latter is just two copies of ${\widehat {HF}}(-\Sigma _L)$.)
 
Gauge-theoretic spectral sequences starting from Khovanov homology
 were constructed by Bloom \cite{Bloom} for   
Monopole Floer homology \cite{KMbook}, and by Kronheimer and Mrowka
\cite{KMspectral} in the context of instanton Floer homology.

It was proved by Baldwin in \cite{Baldwin} that the Heegaard Floer homology spectral sequence is a link invariant. 
The corresponding result for the Monopole Floer homology sequence was given in \cite{Bloom}.

Given a diagram $\DD$ of the link $L$ with $n$ double points the constructions
in \cite{Bloom}, \cite{KMspectral}, and \cite{OSdouble} 
assign higher differentials for the mod 2 Khovanov complex that correspond to
 $k$-dimensional faces of 
$\{0,1\}^n$, with $k\geq 2$. Note that 
 the differentials count
solutions to the
Seiberg-Witten equations for certain $k-1$ dimensional family of metrics in
\cite{Bloom}, instantons in  
\cite{KMspectral}, and certain pseudo holomorphic $k+2$-gons in 
\cite{OSdouble}. 
In particular the maps for a given hypercube depend on some additional data, 
such as 
choices of metric and perturbations, 
or Heegaard diagrams and complex structures.

The construction in the present paper uses the same idea. However 
the extra data is more combinatorial:
At each double point we will fix an orientation of the  arc that connects the two segments of the $0$-resolution, see
Figure \ref{skein}.  This overall choice $\pp$
  is called a decoration of the diagram. 
Given a decoration,  each $k$ dimensional 
face determines a collection of circles in the two-dimensional plane
together with $k$ oriented arcs that connect the circles. These configurations are discussed in Section $2$.  
In Sections $3$ and $4$ we spell out a geometric rule  that assigns non-trivial contributions to certain special
configurations, see for example Figure \ref{2D} and \ref{abcde}. 
In Section $5$ we prove that $\D(\pp) \cdot \D(\pp)=0$, and so we get a chain complex ${\widehat {C}}(\DD, \pp)=(C_\DD, \D(\pp))$
  for a decorated diagram $(\DD,\pp)$.

Another feature of the constrution is that the higher differentials preserve the 
$\delta$ grading in Khovanov homology, and that induces
a $\delta$ grading on  the homology ${\widehat H}(\DD,\pp)$.
 See also  \cite{Greene}, \cite{Manolescu} and \cite{Tweedy} for
discussions on  extra gradings
in Heegaard Floer homology.
There is also a  filtration  on ${\widehat {C}}(\DD,\pp)$, given by the homological grading, and that gives a spectral sequence from 
$Kh(L)$ to ${\widehat H}(\DD,\pp)$.

 In Section
6 we  prove that the homology theory ${\widehat H}$ and the spectral sequence
are well-defined invariants of $L$:

\begin{thm}
Given an oriented link $L$, let $\DD _1$, $\DD _2$ be two diagrams representing $L$, and let $\pp_1$, $\pp _2$ be decorations for
$\DD _1$ and $\DD  _2$ respectively. Then there is a grading preserving isomorpism between 
${\widehat H}(\DD  _1,\pp _1)$ and
${\widehat H}(\DD  _2,\pp _2)$. Furthermore the spectral sequences
 $Kh(L) \longrightarrow {\widehat H}(\DD _1, \pp _1)$ and
$Kh(L) \longrightarrow {\widehat H}(\DD _2, \pp _2)$ 
  are also isomorphic.
\end{thm}

 While the definitions  of ${\widehat H}(L)={\widehat H}(\DD,\pp)$
 and ${\widehat {HF}}(-\Sigma _L\#(S^1\times S^2))$ use different tools, 
the two homology theories are rather similar. In fact it is natural to conjecture 
that the two homology theories are isomorphic as mod 2 graded vector spaces over ${\Z}_2$. 
The similarities are underscored
by a few additional constructions, presented in Section 7. These include a reduced version of ${\widehat H}(L)$
that could be  the  natural counterpart of ${\widehat {HF}}(-\Sigma (L))$, 
and an adaptation of the transverse invariant of \cite{Plam} to ${\widehat H}(L)$, see also 
\cite{Plam2}, \cite{Baldwin}. Computations for the spectral sequence are given in a 
recent work of  Seed, see \cite{Seed}.

While in this paper we work over mod 2 coefficients,  an integer lift of the spectral sequence for 
 odd Khovanov homology is given in
  \cite{SpectralOdd}, see also \cite{Beier}.
In a different direction
it  would be  interesting to compare ${\widehat H}(L)$ with the recent 
work of  Lipshitz, Ozsv\'ath and Thurston, see \cite{LOT}

{\bf Acknowledgments.} I would like to thank John Baldwin, Peter Ozsv\'ath, Cotton Seed and Andr\'as Stipsicz for many helpful 
conversations during the course of this work. This work was 
partially supported by NSF Grant DMS-1006006.

\section{Preliminary constructions}

We will start by recalling  some  constructions 
from Khovanov homology, see \cite{Kh}, \cite{Bn}. Let $L$ be an oriented link, and $\dd$ be a diagram
of $L$ in the plane with $n$ double points. At each crossing we have two 
resolutions $0$ and $1$, see Figure \ref{skein}.
Sometimes it will be helpful to use the
one-point compactification of the plane and view $\dd$ and the resolutions
in the 2-dimensional sphere.   

By ordering the double-points we get an identification  between the 
set of resolutions  ${\mathcal R}$ and $\{0,1\} ^n$. Each resolution 
$I\in {\mathcal R}$ gives a set of disjoint circles $x_1,...,x_t$ in the 
sphere. The resolution $I$ comes equipped with a $2^t$ dimensional vector space $V(I)$ over $\Z_2$. 
It will be useful 
for us to identify the basis of 
$V(I)$ with monomials in $x_i$. This is done by associating a two dimensional vector space $V(x_i)$ 
for each circle with generators $1$ and $x_i$, and 
defining $V(I)$ as the tensor product of $V(x_i)$ for $i=1,...,t$. 
Note that  the original 
construction of Khovanov \cite{Bn}, \cite{Kh} uses different notations,
where $v_-$ plays the role of $x_i$, and $v_+$ the role of $1$.
Finally we define
$$C_{\dd}= \bigoplus _{I\in \rr}V(I)$$ 
 
A $k$ dimensional face of $C_{\dd}$ corresponds $(I,J)\in \rr\times \rr$ with $I<J$ so that
$I$ and $J$ differs at exactly $k$ coordinates.
In Khovanov homology the boundary map is defined by associating maps 
$$ D_{I,J}: V(I)\longrightarrow V(J)$$
to all the 1-dimensional faces (edges) of $\rr$. Our goal is to define 
some higher differentials on ${C}_{\dd}$ by extending the definition 
of $D_{I,J}$ to all $k$ dimensional faces. To this end
we will need to fix some extra data (decoration) at each crossing. 
First note that at each crossing there is an arc connecting the two segments
 of the $0$ resolution. Making surgery along 
this arc produces the $1$ resolution. 
The extra data is an assignment of orientation to all of these arcs. 
The oriented arc at the $i$-th crossing is denoted by $\gamma _i$. 
At each crossing there are two choices. Let $\pp$ denote an overall decoration
 for 
$\dd$.

\begin{figure}
\mbox{\vbox{\epsfbox{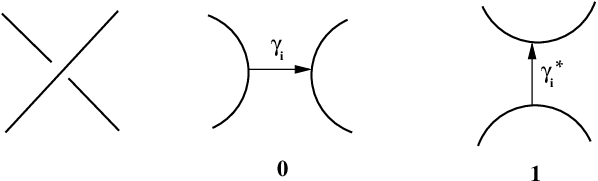}}}
\caption{Resolutions and oriented arcs}
\label{skein}
\end{figure}

A decorated $k$-dimensional face ${\mathcal F}=(I,J,\pp)$ of $(\dd,\pp)$
 determines 
a configuration in the sphere, that consist of the
 circles of $I$ together
with  $k$ oriented arcs $\gamma_{j_1},...,\gamma _{j_k}$, where $j_1,...,j_k$
 are the coordinates where $I$ and $J$ differ. 

Sometimes it will be useful to have a more abstract notion of this 
configuration, that makes no reference to the  knot 
projection $\dd$.

\begin{defn}
A $k$ dimensional {\it configuration} is a set of disjoint circles 
$x_1,...,x_t$ in $S^2$
 together with $k$ embedded oriented arcs $\gamma_1$,...,$\gamma _k$, with the properties that
\begin{itemize}
\item The arcs are disjoint from each other.
\item The endpoints of the arcs lie on the circles.
\item The inside of the arcs are disjoint from the circles.
\end{itemize}
\end{defn}

\begin{defn}
Given a decorated $k$-dimensional face $(I,J,\pp)$ of $\dd$, the corresponding $k$ 
dimensional configuration is denoted by $\cf(I,J,\pp)$.
\end{defn}

\begin{defn}
Given a configuration $\cf=(x_1,...,x_t,\gamma _1,...,\gamma _k)$,  the
\begin{itemize}
\item {\it undecorated configuration} $\cff$ is given by forgetting 
the orientation on the $\gamma$ arcs.
\item   {\it dual configuration} $\cf ^\ast=(y_1,...,y_s,\gamma _1^\ast,...,\gamma _k ^\ast)$ 
is given by the rule that transforms the
$0$-resolution of  Figure  \ref{skein} into the $1$-resolution. 
In particular the dual circles $y_i$ are constructed from the $x$ circles by making surgeries along the $\gamma$ arcs,
 and the dual arcs $\gamma _i ^\ast$ 
are given by rotating $\gamma _i$ by $90$ degrees counter-clockwise.
\item  {\it reverse} $r(\cf)$ is 
defined by reversing 
the orientation for all the $\gamma _i$ arcs. Note that $(\cf ^\ast)^\ast$ is equivalent to $r(\cf)$.

\item {\it mirror} $m(\cf)$ is defined by reversing 
the orientation of the 
two-dimensional sphere. 
\end{itemize}
\end{defn}

Let's call $x_i$ the {\it starting circles} of $\cf$, and $y_i$ {\it the ending circles} of $\cf$. Let 
$$V_0(\cf)=\bigotimes _{i=1} ^t V(x_i),\ \ \ V_1(\cf)=\bigotimes _{j=1}^sV(y_j)$$

Our goal is to define a map
$$\ff_{\cf}:\  V_0(\cf)\longrightarrow V_1(\cf)$$
for each configuration.

\begin{defn}
Given a configuration $\cf=(x_1,...,x_t,\gamma _1,...,\gamma _k)$, those $x_i$  circles that are disjont from all the 
$\gamma$ arcs are called {\it passive} circles. Clearly the same circles are also passive for the dual configuration
$\cf ^\ast$. 
A configuration is called (purely) active, if it has no passive circles. 
By deleting all the passive circles we get the {\it active part} $\cf_0$ of $\cf$. 
The starting circles of $\cf _0$ are called the {\it active starting circles} of $\cf$, and similarly the ending circles
of $\cf _0$ are the {\it active ending circles} of $\cf$. We also have a decomposition
$$V_0(\cf)=V_0(\cf _0)\otimes P(\cf), \ \ V_1(\cf)=V_1(\cf _0)\otimes P(\cf),$$
where $P(\cf)$ is the tensor product of $V(w_i)$ of all the passive $w_i$ circles of $\cf$.
\end{defn}
The map $\ff_{\cf}$ satisfies various properties. We will list these below. 

\begin{defn}{\it Extension Formula.}\label{extension}
For a  configuration $\cf$
the map $\ff _{\cf}$ depends only on the active part $\cf _0$ and the number of passive circles by the following
formula:
$$\ff_{\cf}(a\cdot v)=\ff_{\cf_0}(a)\cdot v,$$
where  $v\in P(\cf)$ and $a\in V_0(\cf)$.
\end{defn}

Recall from \cite{Bn}, \cite{Kh}
 that the  Khovanov differential satisfies the same extension property. 
For one-dimensional faces (edges) there are two kinds of active 
configurations: splitting and joining. 

\begin{itemize}
\item A splitting edge $\cf$, 
 has one active starting circle $x_1$,  
and two active ending circles  $y_1$ and $y_2$. The map 
$\ff_{\cf_0}$ is given by 
$$\ff _{\cf _0}(1)=y_1+y_2,\ \ \ \ff_{\cf _0}(x_1)=y_1y_2.$$
\item A joining edge has two active starting circles, $x_1$ and $x_2$, one active ending circle $y_1$ and we have
$$\ff _{\cf _0}(1)=1,\ \  \ff_{\cf _0}(x_1)=y_1, \ \ 
\ff_{\cf _0}(x_2)=y_1, \ \ \ff _{\cf _0}(x_1x_2)=0.$$
\end{itemize}

\begin{defn}
A configuration $\cf$ is called {\rm disconnected}, 
if the active starting circles of $\cf _0$ can be partitioned into two non-empty sets, 
$c_1,...,c_s$, $d_1,...,d_t$ 
so that none of the  $\gamma$ arcs connect $c_i$ to $d_j$. Otherwise we call the configuration
connected. Note that every $1$-dimensional configuration is connected. 
\end{defn}

\begin{defn}{\it Disconnected Rule.}\label{disconnected}
If  $\cf$ is a disconnected configuration 
 then 
$$\ff_{\cf} \equiv 0.$$
\end{defn}

\begin{defn}{\it Conjugation Rule.}\label{conjugation}
For each configuration $\cf$ we have
$$\ff_{\cf}= \ff_{r(\cf)}$$
\end{defn}

\begin{defn}{\it Naturality Rule.}\label{naturality}
Let $\cf$ and $\cf'$ be two  $k$ dimensional configurations, with the property that there is an
orientation preserving diffeomorphism of the sphere that maps $\cf$ to $\cf'$ . Then the diffeomorphism 
induces natural identifications
 $V_0(\cf)=V_0(\cf')$, 
and $V_1(\cf)=V_1(\cf ')$. Under these identifications we have
$$\ff_{\cf}= \ff_{\cf'}$$
\end{defn}

Now we discuss the duality rule. First note that the monomials in $V_0(\cf)$
 and $V_1(\cf)$ give natural basis for these vector spaces. In particular
if $a$ is a monomial in $V_0(\cf)$ then we can write
$$\ff_{\cf}(a)=\sum _b \alpha(a,b)\cdot b$$
where $\alpha(a,b)\in  \Z _2$ and the sum is over all the monomials
$b\in  V_1(\cf)$. We will call $\alpha(a,b)$ the coefficient of
 $\ff _{\cf}(a)$ at $b$. Given a circle $z$ we define the 
duality map on the $2$ dimensional
vector space $V(z)$ by $1^\ast =z$, $z^\ast=1$. 
This induces the duality maps on  $V_0(\cf)$ and $V_1(\cf)$.

\begin{defn}{\it Duality Rule.}\label{duality}
Let $\cf$ be a  configuration, and $m(\cf^\ast)$ 
the mirror of the dual configuration. 
Then for all pairs of monomials $(a,b)$, $a\in V_0(\cf)$, $b \in V_1(\cf)$ 
the coefficient of $\ff_{\cf}(a)$ at $b$ is equal to the
coefficient of $\ff_{m(\cf ^\ast)}(b^\ast)$ at $a^\ast$.
\end{defn}

\begin{defn}{\it Filtration rule.} \label{filtration}
Let $\cf$ be a configuration, $a\in V_0(\cf)$, $b\in V_1(\cf)$ monomials. 
For a point  $P$  in the union of the 
starting circles, let $x(P)$ and $y(P)$ denote the starting and ending 
circles that go through $P$. If $a$ is divisible by $x(P)$ and the
 coefficient of $F_{\cf}(a)$ at $b$ is non-zero, then $b$ is divisible by 
$y(P)$.
\end{defn}

The next property involves the grading shift of $F$. For $a\in V(x)$ 
define 
$$gr(1)={1},\ \ \ \ gr(x)=-{1}$$
For monomials  $a\in V_0(\cf)=\otimes _{i=1}^t V(x_i)$, $b\in V_1(\cf)=
\otimes _{j=1}^s V(y_j)$ define the grading to be the sum of the gradings
in each factor. 

\begin{defn}{\it Grading rule.}\label{grading}
Let $\cf$ be a $k$-dimensional configuration, $a\in V_0(\cf)$,
$b\in V_1(\cf)$ monomials. If the coefficient of $F_{\cf}(a)$ at $b$ is
 non-trivial, then 
$$gr(b)-gr(a)=k-2.$$
\end{defn}

Note that for $1$-dimensional configurations $F_{\cf}$ satisfies the rules in 
Definition \ref{disconnected}- \ref{grading}

\section{$2$-dimensional configurations}

In this section we will define the $F$ map for all  $2$-dimensional configurations.
 According to Definition \ref{extension}
it is enough to spell out the rule for the active 
part of $2$-dimensional configurations.  Furthermore 
according to Definition \ref{disconnected}, \ref{conjugation} and \ref{naturality} 
 it is enough to consider connected  configurations 
modulo orientation preserving diffeomorphisms in the sphere, and reversals
$\cf \rightarrow r(\cf)$. 
 The resulting equivalence classes are listed in Figure 2.
The rules for the active map 
$$F=F_{\cf _0}:\ V_0(\cf_0)\longrightarrow V_1(\cf _0)$$ are given in
Definition \ref{2Drule}, where we use monomials in the active starting circles 
$x_i$ as a basis of $V_0(\cf_0)$ and  list only the non-zero terms of $F$.

\begin{defn}\label{2Drule}
\begin{itemize}
\item For a Type 1 configuration 
$$\ff(1)=1.$$

\item For a Type 2 or Type 3 configuration there are three starting circles and one ending circle.  
If $x_1$ denotes the starting circle  that meets both $\gamma$ arcs, then
$$\ff(x_2x_3)=y_1.$$

\item For a Type 4 or Type 5 configuration,
$$\ff(1)=y_1,$$ where among the three ending circles,  $y_1$ denotes the  
unique circle that meets both of the dual $\gamma ^\ast$ arcs.  

\item For a Type 6 or Type 7 configuration there are two starting and two ending circles.
Let $x_1$ denote the starting circle that meets
 both of the $\gamma$ arcs, and $y_1$ denotes the ending  circle that meets both
 of the dual arcs. Then
$$\ff(x_2)=y_1$$

\item For a Type 8 configuration
$$\ff(1)=1,\ \ \ff(x_1)=y_1.$$

\item For a Type 9, configuration 
$$\ff(x_1x_2)=y_1y_2.$$

\item Finally for a Type 10, 11, 12, 13, 14, 15, 16 or for a disconnected
 configuration 
$$\ff\equiv 0.$$
\end{itemize}
\end{defn}

\begin{figure}
\mbox{\vbox{\epsfbox{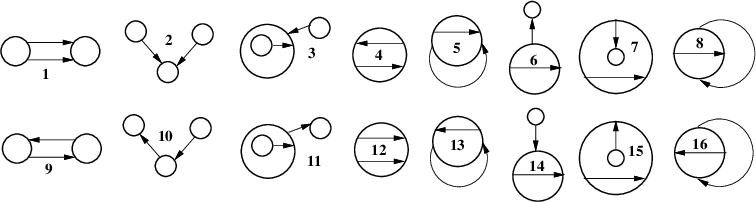}}}
\caption{The classification of active, 
connected $2$-dimensional configurations in the sphere,  modulo the additional
relation  that $\cf$ is equivalent to  $r(\cf)$.}
\label{2D}
\end{figure}

\begin{lemma}\label{2Dcheck}
For $2$-dimensional $\cf$ configurations the map 
$\ff_{\cf}$ satisfies the rules in 
Definition \ref{disconnected}-\ref{grading}.
\end{lemma}

\begin{proof} 
The rules in Definition \ref{disconnected}, \ref{conjugation}, 
\ref{naturality}, \ref{filtration}, \ref{grading} follow immediately from
the definition of $F_{\cf}$. 

According to Definition \ref{extension}, it is enough to 
check the duality rule for the active 2-dimensional configurations. 
If the configuration $\cf$ is disconnected,
then the dual configuration is also disconnected, 
so $\ff_\cf=0$, $\ff_{m(\cf ^\ast)}=0$. 

After dividing with the relation that $\cf$ is equivalent to its reverse, 
the connected types of Figure \ref{2D} are related 
 in the following way, $1^\ast=9$,
$2^\ast=4$, $3^\ast =5$, $6^\ast =14$, $7^\ast= 15$, $8^\ast =16$, 
$10^\ast =12$, $11^\ast= 13$.
Furthermore $m(i)=i$ for $1\leq i\leq 5$, 
or $9\leq i\leq 13$,  and $m(6)=14$,$m(7)=15$, $m(8)=16$. 
It follows that $m(1^\ast)=9$, $m(2^\ast)=4$, 
$m(3^\ast)=5$, $m(6^\ast)=6$, $m(7^\ast)=7$, $m(8^\ast)=8$, 
and $m(i^\ast)$ is of type $j$ with $10\leq j\leq 16$ 
if and only if $10\leq i\leq 16$. Checking the duality formula is now rather 
straightforward. For example if $\cf$ is of  Type $6$ then $F(x_2)=y_1$,
$x_2^\ast =x_1$, 
$y_1^\ast=y_2$, and indeed for the $m(\cf ^\ast)$ configuration (which is also
of Type $6$) the circle
$y_2$ is mapped to $x_1$. The other $2$-dimensional configurations
 are left for the interested reader to check.
\end{proof}

\section{Contributions of $k$-dimensional configurations.} \label{higher}

Given $k>2$ we will distinguish 5 kinds of $k$ dimensional configurations for which $\ff_{\cf}\neq 0$.

\begin{defn}
The configuration  $\cf=(x_1,...,x_s,\gamma _1,...,\gamma _k)$ 
is of Type $A_k$ 
if  for each pair $(i,j)$
with $1\leq i,j\leq k$ the two dimensional configuration
$(x_1,...,x_s,\gamma _i,\gamma _j)$ is of  type $1$, see Figure \ref{abcde}. 
It follows
 that $\cf$ has two active starting circles and $k$ active
ending circles.
For a Type $A_k$ configuration we define 
$$\ff_{\cf _0}(1)=1.$$
\end{defn}

\begin{defn}
A $k$-dimensional configuration $\cf$ is of Type $B_k$ 
if $m(\cf ^\ast)$ is of Type $A_k$. It follows that $\cf$ has $k$ active 
starting circles and two active ending circles. For a Type $B_k$ configuration
define
$$\ff_{\cf _0}(\prod _{i=1} ^k x_i) =y_1 y_2 $$
\end{defn}

\begin{figure}
\mbox{\vbox{\epsfbox{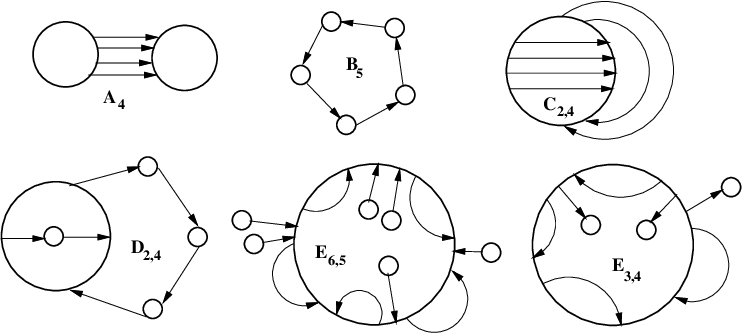}}}
\caption{ Examples for $A$, $B$, $C$, $D$, and $E$ types of configurations}
\label{abcde}
\end{figure}

\begin{defn}\label{Cpq}
Let $\cf$ be a $k$-dimensional configuration,
 with the property that it has only 
one active starting circle $x_1$. 
This circle separates the sphere  into two regions, and the set of arcs 
decomposes as
$$\{\gamma _1,...,\gamma _k \}= \{e_1,...,e_p\} \cup \{ f_1,...,f_q\}$$
where $e_i$ lie in one side of $x_1$ and $f_j$ lie on the other side.
Then $\cf$ is of Type $C_{p,q}$ if
$p\geq 1$, $q\geq 1$ and for each $(i,j)$ pair with
$1\leq i\leq p$, $1\leq j\leq q$ 
the two dimensional configuration $(x_1,e_i,f_j)$ is 
of Type $8$. For a Type $C_{p,q}$ configuration we define
$$\ff_{\cf _0}(1)=1.$$
\end{defn}

Note  that $p+q=k$ and $\cf$ has $k-1$
 active ending circles. Furthermore if $\cf$ is Type $C_{p,q}$ then 
$r(\cf)$ is also Type $C_{p,q}$.

\begin{defn}\label{Dpq}
A $k$-dimensional configuration $\cf$ is of Type $D_{p,q}$ 
if  and only if $m(\cf ^\ast)$  is of type
$C_{p,q}$. In this case we define 
$$\ff_{\cf _0}(\prod _{i=1} ^{k-1} x_i)= y_1.$$
\end{defn}

\begin{defn}
A $k$-dimensional configuration $\cf=(x_1,...,x_s,\gamma _1,...,\gamma _k)$
 with $p+1$ active starting circles, and $q+1$ active ending circles, 
is called of type 
$E_{p,q}$, if for each pair $(i,j)$ with $1\leq i,j\leq k$
the $2$-dimensional configuration $(x_1,...,x_s, \gamma _i,\gamma_j)$
 is of 
type, $2$, $3$, $4$, $5$, $6$, or $7$. See Figure \ref{abcde} for examples
of $E_{6,5}$ and $E_{3,4}$. Note that in a type $E_{p,q}$ 
configuration
there is a unique starting circle $x_1$, called the 
{\it central starting circle},
with the property that it meets all the $\gamma $ arcs. The other active
$x_i$ are called  degree $1$ starting circles.
Similarly there is a unique  ending circle $y_1$ with the 
property that $y_1$ meets all the $\gamma ^\ast$ arcs. This circle is called 
central ending circle, and the other active $y_i$ are called the dual degree
$1$ circles.
Using this labeling, we define
$$\ff_{\cf _0}(\prod _{i=2}^{p+1} x_i)=y_1,$$
when $p\geq 1$ and 
$$\ff_{\cf_0}(1)=y_1,$$
when $p=0$.
\end{defn}

Note that if a $k$ dimensional configuration $\cf$ is Type $E_{p,q}$, then
$k=p+q$, $r(\cf)$ is Type $E_{p,q}$ and $m(\cf ^\ast)$ is of Type $E_{q,p}$.

\begin{lemma}\label{rules}
For any  $k$-dimensional $\cf$ configurations with $k\geq 1$ the map 
$\ff_{\cf}$ satisfies the rules in 
Definition \ref{disconnected}-\ref{grading}.
\end{lemma}

\begin{proof}
This follows immediately 
from the above definitions and remarks and Lemma \ref{2Dcheck}
\end{proof}

It is also helpful to revisit now the contribution of the $2$-dimensional
configurations. In later calculations we will 
 think of  them as special cases of the $A,B,C,D,E$ types. 
Note that  
for a two-dimensional configuration, Type $2$ and $3$ of Figure \ref{2D}
are examples for
$E_{2,0}$, Type $4$ and $5$ are $E_{0,2}$,
and Type $6$ and $7$ are $E_{1,1}$. Similarly Type $1$ is $A_2$ and  Type $9$ is 
$B_2$.

There is also a special case, when $\cf$ is of Type $8$, since then  $\cf$ 
is in a sense both  $C_{1,1}$ and $D_{1,1}$. 
In fact, recall that $\ff_{\cf _0}$ has two nontrivial terms: 
$$\ff_{\cf _0}(1)=1 \ \ ,\ \ \ \ \ff_{\cf _0}(x_1)=y_1 $$
and the first term corresponds to the $C_{p,q}$ rule, see Definition \ref{Cpq},
 and the second to the $D_{p,q}$ rule of Definition \ref{Dpq}.

\section{The definition of $\D$. }
In this section we define $\D$ and study the role of decorations.

Given a diagram $\dd$ with decoration $\pp$ and a $k$-dimensional face 
$(I,J)$ the corresponding configuration is denoted by $\cf(I,J,\pp)$.
Clearly  $V(I)$ is naturally identified with $V_0(\cf(I,J,\pp))$ and 
$V(J)=V_1(\cf(I,J,\pp))$.

Using these identifications, we  define 
$$D_{I,J,\pp}:\ V(I)\longrightarrow V(J)$$
by the formula 
$$D_{I,J,\pp}=F_{\cf(I,J,\pp)}.$$

\begin{defn} Let $n$ denote the number of double-points in $\dd$. 
For $1\leq k\leq n$ we
define  
$$\D_k(\pp):\ C_{\dd}\longrightarrow C_{\dd}$$
as the sum of $D_{I,J,\pp}$ for all $k$-dimensional
faces $(I,J,\pp)$. Now
  the boundary map
$$\D(\pp): C_{\dd}\longrightarrow C_{\dd}$$
is defined by
$$\D(\pp) =\sum _{k=1}^n \D_k(\pp).$$
\end{defn}

Note that $\D_1(\pp)$ agrees with the Khovanov differential, and in particular 
$\D_1(\pp)$ doesn't depend on $\pp$.

\subsection{The $H_m$ maps.}

We define an ``edge-homotopy'' for one dimensional configurations.
Similarly to the $F$ maps, $H$ is defined on the active part $\cf _0$ and 
then extended to
$$ H_{\cf}:\ V_0(\cf) \longrightarrow V_1(\cf)$$
by the extension formula of Definition \ref{extension}.

\begin{defn}
If $\cf$ is a one-dimensional configuration, then
$$H_{\cf _0}:\ V_0(\cf _0)\longrightarrow V_1(\cf _0)$$ is defined: 
\begin{itemize}
\item For a splitting edge
$$H_{\cf _0}(1)=1.$$
\item For a joining edge
$$H_{\cf _0}(x_1x_2)=y_1.$$
\end{itemize}
Furthermore
$$H_{\cf}(a\cdot v)=H_{\cf _0}(a)\cdot v,$$
for $a\in V_0(\cf _0)$, $v\in P(\cf)$.
\end{defn}

Note that $H$ doesn't depend on the orientation of the $\gamma$ arcs.

\begin{defn}
For a $1$-dimensional face $(I,J)$ we define 
$$H_{I,J}:\ V(I)\longrightarrow V(J)$$ by the formula
$$H_{I,J}=H_{\cf(I,J)}.$$ 
For the $m$-th double point in the diagram $\dd$ we define 
$$H_m:\ C_{\dd}\longrightarrow C_{\dd}$$
by summing $H_{I,J}$ over those $1$-dimensional faces, where $I$ and $J$ differ only in the $m$-th coordinate.

\end{defn}

\subsection{Dependence on the perturbations.}

\begin{thm}\label{Relation}
Suppose that $\pp$ and $\pp '$ are decorations of the diagram $\dd$ that differ only at 
the $m$-th crossing. 
Then $\D (\pp)$ and $\D(\pp')$ 
are related by the following formula:
$$\D(\pp')=\D(\pp)+ H_m\cdot \D(\pp)+\D(\pp)\cdot H_m .$$
\end{thm}

\begin{figure}
\mbox{\vbox{\epsfbox{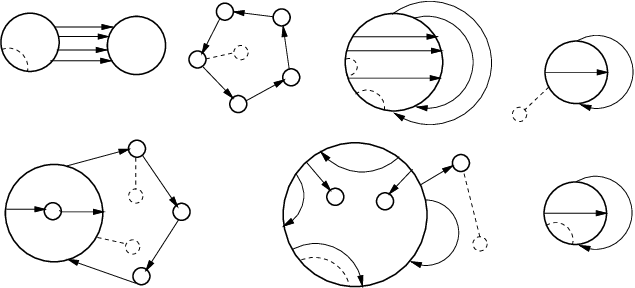}}}
\caption{  $\delta$ arcs and $w$ circles.}
\label{DH}
\end{figure}

\begin{proof}
Let $\delta$ denote the 
(unoriented) arc that correspond to the $m$-th crossing.
For a $k$-dimensional face $(I,J)$ with $I(m)=0$, $J(m)=1$,
let's define $I'$ and $J'$ 
as  $I'(m)=1$, $J'(m)=0$, $I'(i)=I(i)$, $J'(i)=J(i)$ for $i\neq m$. 
Then Theorem \ref{Relation} is equivalent to the following equation for
all these 
$k$ dimensional $(I,J)$ faces:

$$ (1)\ \ D_{I,J,\pp}-D_{I,J,\pp'}=H_{I,I'}\cdot D_{I',J}+D_{I,J'}\cdot H_{J,J'}.$$

Of course, both sides depend only on the equivalence class
of the configuration $\cf=\cf(I,J,\pp)$ 
and the position of the $\delta$ arc among the $\gamma$ arcs in
$\cf$. Note also that $\pp$ and $\pp'$ agree on the $k-1$ 
dimensional faces $(I,J')$ and $(I',J)$
so we can safely omit 
that from the notation.

The rest of this section is devoted to the proof of Equation $(1)$.
Clearly it  is enough to check the  equation  for active configurations.
Another observation is that both the $D$ and the $H$ maps satisfy the duality rule, 
 in particular, 
if Equation  $(1)$ holds for the pair
$(\cf,\delta)$ then it holds for $(m(\cf ^\ast),\delta ^\ast)$.

{\bf $\cf$  is a disconnected configuration.} 
By the disconnected rule we have $D_{I,J,\pp}=D_{I,J,\pp '}=0$. Furthermore by extension formula
for $\D$ and $H$ we have
$$H_{I,I'}\cdot D_{I',J}=D_{I,J'}\cdot H_{J',J}.$$

{\bf $\cf$ is $2$-dimensional.}
In this case  checking Equation $(1)$ is an easy exercise. We  illustrate this 
when  $\cf$ is of Type $1$. Let $x_1$, $x_2$ denote the circles of $I$, $y_1$, $y_2$ denote the circles of $J$, and
$w_1$ the circle of $I'$. Since $(I,I')$ is a join cobordism, 
$$H_{I,I'}(x_1x_2)=w_1, \ \ D_{I',J}(w_1)=y_1y_2.$$
For the other decomposition of $(I,J)$ we have a non-trivial composition by
$$D_{I,J'}(1)=1\ , \ H_{J',J}(1)=1.$$
It follows that the right hand side of Equation (1) maps 
$1$ to $1$, and $x_1x_2$ to $y_1y_2$. 
Since $\cf(I,J,\pp)$ is of Type $1$  and $\cf(I,J,\pp ')$ is of Type $9$ the nontrivial terms of $D$ are
$$D_{I,J,\pp}(1)=1,\ \ D_{I,J,\pp'}(x_1x_2)=y_1y_2,$$
 and Equation $(1)$ holds. 
We leave 
the rest of the $k=2$ cases to the reader. 

For the connected  $k\geq 3$ cases, we start with:

{\bf $D_{I,J,\pp}$ or $D_{I,J,\pp '}$ is non-trivial.}
We can then assume $D_{I,J,\pp}\neq 0$, so $\cf (I,J,\pp)$ is of  Type $A_k$,$B_k$, $C_{p,q}$, $D_{p,q}$ or $E_{p,q}$ with $p+q=k$.
Since $k\geq 3$ it follows that changing the orientation of one arc gives a configuration with
trivial contribution. In particular  $D_{I,J,\pp'}=0$. We 
claim that in each of these case
either
$$ (i)\ \ H_{I,I'}\cdot D_{I',J}=0, \ D_{I,J'}\cdot H_{J',J}=D_{I,J,\pp}$$
or
$$ (ii)\ \ H_{I,I'}\cdot D_{I',J}=D_{I,J,\pp}, \ \ D_{I,J'}\cdot H_{J',J}=0$$
holds. 

If $\cf(I,J,\pp)$ is of Type: 
\begin{itemize}
\item $A_k$, then  $\cf(I,J')$ is $A_{k-1}$ and
$(i)$ holds.

\item $C_{p,q}$ with $p\geq 2$, $q\geq 2$, then $\cf(I,J')$ is
$C_{p,q-1}$ (or $C_{p-1,q}$) and  $(i)$ holds.

\item $C_{p,1}$, with $p\geq 3$, and $\delta$ 
is one of $e_1,...,e_p$ then $\cf(I,J')$ is  $C_{p-1,1}$ and
$(i)$ holds.

\item $C_{p,1}$ with $p\geq 2$ and  $\delta$ equals to  $f_1$, then  
 $(I',J)$ is of Type $A_{p}$
and $(ii)$ holds.

\item $C_{2,1}$ and $\delta $ is one of $e_1,e_2$, then $\cf(I,J')$ is of Type $8$
and $(i)$ holds.

\item Type $E_{p,q}$, then  $\delta$ either connects the central starting circle
$x_1$ 
to $x_t$ for $2\leq t\leq p+1$, or connects $x_1$ to itself. 
In the first case $\cf(I,J')$ and $\cf(I',J)$ are  $E_{p-1,q}$ and
$(i)$ holds.
In the second case $\cf(I,J')$ and $\cf(I',J)$ are  $E_{p,q-1}$ and
$(ii)$ holds.
\end{itemize}

When $\cf$ is of Type $B_k$ or Type $D_{p,q}$, then Equation $(1)$ follows from 
the earlier duality argument.

It remains to check the case when:

{\bf Both $D_{I,J,\pp}$ and $D_{I,J,\pp'}$ are trivial.}
Using the duality property of $D$ and $H$, it is enough to consider the case  
$$D_{I,J'}\cdot H_{J',J}\neq 0.$$ Now we have to show that
$$ (iii)\ \ H_{I,I'}\cdot D_{I',J}= D_{I,J'}\cdot H_{J',J}$$
holds. We list all the cases and configurations, using the Type of $\cf(I,J')$ to determine all the possibilities: 
If $\cf(I,J')$ is of Type

\begin{itemize}
\item $A_{k-1}$, $k\geq 3$ then $D_{I,J'}(1)=1$ and
$(J',J)$ needs to be a splitting edge. Let $x_1$ and $x_2$ denote the active circles 
in $\cf(I,J)$. If the  $\delta$ connects $x_1$ and $x_2$ then either $(I,J,p)$ or $(I,J,p')$ is of type $A_k$
and that contradicts our assumption. It follows that $\delta$ connects $x_1$ to itself or $x_2$ to itself, see Figure
\ref{DH}.
In these cases
$\cf(I',J)$ and $\cf(I,J')$ are both type $A_{k-1}$ and $(iii)$ holds.

\item $B_{k-1}$, $k\geq 3$ then $D_{I,J'}\cdot H_{J',J}\neq 0$ implies that 
$(J',J)$ is a joining edge. If $\delta$ connects $y_1$ and $y_2$, then either
$\cf(I,J,\pp)$ or $\cf(I,J,\pp')$ is of type $C_{k-1,1}$ (a case covered earlier). It follows that $\delta$ connects another circle
$w$ to $y_1$ or $y_2$, see Figure \ref{DH}, and $\cf(I',J)$ is of Type $B_{k-1}$.

\item $C_{p,q}$ with $p+q \geq 3$, then $D_{I,J'}\cdot H_{J',J}\neq 0$ implies that
$(J',J)$ is a splitting edge, see Figure \ref{DH} for the possible choices for $\delta$. In all cases $\cf(I',J)$ is of Type
$C_{p,q}$. 

\item $D_{p,q}$ with $p+q \geq 3$ then $(J',J)$ has to be a joining edge. Since $\cf(I,J')$ has only one active ending circle $y_1$, it follows
that $\cf(I,J)$ has an additional active starting circle $w$. See Figure \ref{DH} for the possible choices for $\delta$ and $w$.

\item $E_{p,q}$  and $(J',J)$ is splitting edge, then $\delta$ has to split one of $y_2,...,y_{q+1}$. It follows that
$\cf(I',J)$ is also of Type $E_{p,q}$.

\item $E_{p,q}$ and $(J',J)$ is a join edge, then $\delta$ has to connect a new circle $w$ to $y_1$. If $\delta$ attaches $w$ to portion of $y_1$ that lies in the central starting circle $x_1$, then either
$\cf(I,J,\pp)$ or $\cf(I,J,\pp')$ is of Type $E_{p+1,q}$. It follows that $\delta$ attaches $w$ to the portion of $y_1$ that
lies in $x_i$ for some $2\leq i\leq p+1$, see Figure \ref{DH}. 
It follows that $\cf(I',J)$ is of Type 
$E_{p,q}$.

\item $8$ and $(J',J)$ is a join edge, then $\delta$ connects a new circle
$w$ to $y_1$, see Figure \ref{DH}, and $\cf(I',J)$ is also of type 8. 

\item $8$ and $(J',J)$ is a split edge then in $\cf(I,J,\pp)$ there is a dual circle that is only attached to 
$\delta ^\ast$ and none of the other $\gamma$ arcs, see Figue \ref{DH},
 (otherwise either $\cf(I,J,\pp)$ or $\cf (I,J,\pp')$ would be of type 
$C(2,1)$). It follows that $C(I',J)$ is also of type $8$.   

\end{itemize}

Checking $(iii)$ is straightforward in all the cases.
\end{proof}

\section{Proof of $\D(\pp)\cdot \D(\pp)=0$.}

In this section (decorated) 
configurations are denoted as $\cf$ or $(\cff,\pp)$ and undecorated 
configurations are denoted by $\cff$.

\begin{thm}\label{dd}
For every $k$-dimensional configuration $(\cff, \pp)$ we have 
$$   \sum _{i=1} ^{k-1} \D_i (\pp)\cdot \D_{k-i}(\pp)=0$$
\end{thm}

\begin{proof}
We will use induction on $k$. The $k=2$ case is trivial, since 
$\D_1(\pp)$ is the Khovanov differential. 

\begin{lemma}\label{ind}
Let $\pp$ and $ \pp'$ be two decorations   on the $k$ dimensional undecorated configuration
$\cff$. If Equation $(2)$ holds for all $k-1$ dimensional configurations, then 
$$
\sum _{i=1}^{k-1} \D_{k-i}(\pp)\Big( \D_{i}(\pp)(a)\Big)=
\sum _{i=1}^{k-1} \D_{k-i}(\pp')\Big( \D_{i}(\pp')(a)\Big) .$$
for all $a\in V_0(\cf)$.
\end{lemma}

\begin{proof} It is enough to consider the case when $\pp$ and $\pp'$ differ at a single crossing indexed by $m$. 
According to Theorem \ref{Relation}, we have
$$\D _j(\pp ')= \D_j (\pp)+\D _{j-1}(\pp)\cdot H_m + 
H_m\cdot \D _{j-1}(\pp). $$

This formula together with the trivial observations:
$$H_m\cdot H_m=0,\ \ \ \ \ H_m \cdot \D_j\cdot H_m = 0$$
finishes the argument.
\end{proof}

The strategy to prove Theorem \ref{dd}
 is to consider undecorated configurations
$\cff$ 
and monomials $a\in V_0(\cff)$ and $b \in V_1(\cff)$. We will need the 
following:

\begin{thm}\label{ki}
Let $\cff$ 
be an undecorated $k$ dimensional configuration 
with $k\geq 3$,  let $a$ and $b$  denote  monomials $a\in V_0(\cff)$, $b\in V_1(\cff)$.  
For every triple, $(\cff, a, b)$ 
 there exists a decoration $\pp$ on $\cff$ so that the coefficient of
$$ \sum _{i=1}^{k-1} \D _{k-i}(\pp)\Big(\D _{i}(\pp)(a)\Big)$$
at $b$ equals to $0$. 
\end{thm}

Theorem \ref{dd} follows immediately from Theorem \ref{ki} and Lemma \ref{ind} by induction on
$k$.
\end{proof}

The rest of this section is devoted to the proof of Theorem \ref{ki}. 
Let's start with a few notations and remarks. If for a given $\cff $ 
the statement in Theorem \ref{ki} holds for all $a\in V_0(\cff)$ and $b\in 
V_1(\cff)$, we say that Theorem \ref{ki} holds for $\cff$. Similarly if for
a given $(\cff,a)$  the 
statement holds for all $b\in V_1(\cff)$, we say that Theorem \ref{ki} holds
for $(\cff,a)$.

Note that  $\D _{k-i}(\pp) (\D _{i}(\pp)(a))$ can be written by summing 
$$F_{\cf(2)}(F_{\cf(1)}(a)), $$ over all 
decompositions of $\cf$ as $\cf= \cf (1)\ast \cf(2)$, with $\cf (1)$ 
having
 dimension $i$,  and $\cf (2)$  dimension $k-i$.  
Clearly $\cf(1)$ and $\cf(2)$ 
is determined  by the decomposition of the 
$\gamma$ index set $\{1,2,...,k\}$ as 
the union of sets $U$ and $V$, where $|U|=i$ and $|V|=k-i$, where
$U$ corresponds to $\cf(1)$, $V$ corresponds to $\cf(2)$. 

In proving Therorem \ref{ki}  note that according the extension property it is enough to consider 
active configurations. Next we consider the case when
 $\cff$ is disconnected. If the active part of 
$\cff$ has more than $2$ connected components then either $F_{\cf(1)}$ or 
$F_{\cf(2)}$ is $0$  by the disconnected rule. 
If there are $2$ connected 
components then there are still two decompositions of the $k$ dimensional cube to consider.
However their contributions agree according to the extension formula. 

From now on we will assume that $\cff$ is both active and connected. (Note that $\cf(1)$ or $\cf(2)$ could 
be disconnected, or could have passive circles.)

Now we  discuss a few moves on $\cff$ in order 
to cut down the number of cases to consider:

\begin{lemma}\label{dual2}
The statement in Theorem \ref{ki} holds for the triple
$(\cff,a,b)$ if and only if it holds for $(m(\cff ^\ast), b^\ast, a^\ast)$.
\end{lemma}

\begin{proof}
This follows immediately from the duality rule.
\end{proof}

\begin{lemma}\label{rot1}
Let $\cff$ and $\cff '$ be configurations that differ by a rotation move of 
Figure \ref{rotation}. Then Theorem \ref{ki} 
holds for $(\cff,a,b)$ if and only if it holds
for $(\cff ',a,b)$. 
\end{lemma}

\begin{proof}
The second row of Figure \ref{rotation} indicates how to modify the decorations. 
Using these decorations all the maps in $(\cff,\pp)$ agree with the maps in $(\cff',\pp')$.
\end{proof}

\begin{figure}
\mbox{\vbox{\epsfbox{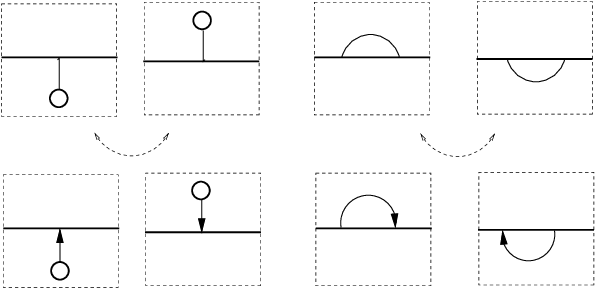}}}
\caption{ Rotation}
\label{rotation}
\end{figure}

\begin{lemma}\label{cf1}
Suppose that $\cff$ contains an active starting circle $x_1$ that 
meets only one of $\gamma$ arcs, say $\gamma _1$. We will call $x_1$ a degree
$1$ starting circle. 
Let $x_2$ denote the other circle that meets $\gamma _1$, and let
$a\in V_0(\cff)$ be a monomial.

\begin{itemize}
\item If  $a$ 
  is not divisible by $x_1$ then Theorem \ref{ki} holds
for $(\cff,a)$.
\item If $a$ is divisible by $x_2$, then
 Theorem \ref{ki} holds
for $(\cff,a)$.
\end{itemize}
\end{lemma}

\begin{proof}
We will use the notation of $\cf(1)$, $\cf(2)$, $U$ and $V$ as above. If $a$ is not divisible by $x_1$ then there are only $2$ possible
decompositions 
with $F_{\cf(2)}(F_{\cf(1)}(a))\neq 0$, corresponding to 
$U=\{ 1\}$ and $U'=\{2,....,k\}$. However these terms cancel
each other according to the extension formula. 

For the second part, it is enough to consider the case when 
$a$ is divisible by $x_1x_2$. It follows that $F_{\cf (1)}(a)=0$ if $1\in U$. If $1\notin U$, let $w$ denote
the ending circle of $\cf(1)$ that contains the intersection point between 
 $\gamma _1$ and  $x_2$. According to the 
filtration rule  $F_{\cf(1)}(a)$ is divisible by $x_1\cdot w$. This implies that its image under $F_{\cf(2)}$ 
is trivial.
\end{proof}

\begin{lemma}\label{cf2}
Supposes that $\cff$ contains an active starting circle $x_1$ that is connected to the other circles by  two of the 
$\gamma$ arcs $\gamma _1$ and $\gamma _2$,
see Figure \ref{deg2}. Let $a\in V_0(\cff)$ denote a monomial, and  $x$, 
$x'$ denote the other circles that meet $\gamma _1$ and $\gamma _2$. If $x=x'$ define $p=x$. Otherwise define $p=x\cdot x'$
\begin{itemize}
\item If $a$ is not divisible by $x_1$, then  Theorem \ref{ki} holds
for $(\cff,a)$.
\item If $a$ is divisible with $p\cdot x_1$
then  Theorem \ref{ki} holds
for $(\cff,a)$.
\end{itemize} 
\end{lemma}

\begin{proof}
For the first statement use the decoration as in the center of Figure \ref{deg2}. Consider 
decompositions with $F_{\cf(2)}(F_{\cf(1)}(a))\neq 0$. If $a$ is not divisible by
$x_1$ then there are only two possible non-trivial terms, corresponding to decompositions with
$U=\{1\}$ and $U'=\{2\}$. 
However according to the chosen decoration $\cf(2)$ and $\cf(2)'$ are equivalent configurations. 
Since $a$ is not divisible by $x_1$ we also have $F_{\cf(1)}(a)=F_{\cf(1) '}(a)$. It follows that the  
contributions  cancel each other. 

For the second part use the decoration as in the right of Figure \ref{deg2}. Suppose that
$F_{\cf(1)}(a)\neq 0$. If $\{1,2\}\subset U$, then  $a$ being divisible with $p\cdot x_1$  implies that
$\cf(1)$ is Type $B$ or $D$, but that contradicts the choice of decoration. If $1\in U$, $2\in V$ then
 $\cf(1)$ has to be a join edge or Type $E$, 
but that contradicts $F_{\cf}(a)\neq 0$ and the choice of $a$. The case $1\in V$, $2\in U$ is ruled out 
the same way. It remains to consider $\{1,2\}\subset V$. In this case the filtration rule implies
that $F_{\cf(1)}(a)$ is divisible by $x_1$ and the  other (one or two) ending circles of $\cf(1)$
 that meets $\gamma _1$ and $\gamma _2$. So $F_ {\cf(2)}(F_{\cf(1)}(a))\neq 0$, implies that $\cf(2)$ is  Type $B$ or $D$.
However that again contradicts the chosen decoration.
\end{proof}

\begin{figure}
\mbox{\vbox{\epsfbox{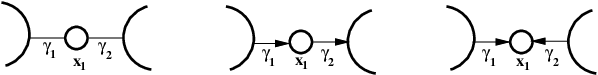}}}
\caption{}
\label{deg2}
\end{figure}

\begin{lemma}\label{trad2}
Let $\cff$ and $\cff'$ be configurations that differ by a trading move of 
Figure \ref{trading}. Then Theorem \ref{ki} holds for $\cff$ if and only if it holds 
for $\cff '$. 
\end{lemma}

\begin{proof}
Let $\cff$ denote the local configuration in the far left of Figure \ref{trading}, and  $y_1$ denote the ending circle of $\cff$ that is given by merging $x_1$ and $x_2$. Also let
$y_1'$ denote the ending circle of $\cff '$ that meets only $
(\gamma _1 ')^\ast$, and $y_2$ be the other ending circle of $\cff'$ that meets $
(\gamma _1 ')^\ast$. Let's write $V_0(\cff)= V(x_1)\otimes V(x_2)\otimes W$ and $V_1(\cff)=V(y_1)\otimes P$.
Then 
$V_0(\cff ')= V(x_2')\otimes W$ and $V_1(\cff ')=V(y_1 ')\otimes V(y_2 ')\otimes P$.
 According to Lemma \ref{cf1} and the filtration rule, it is enough to consider $(\cff, a,b)$   where 
$a=x_1\cdot w$ and  $b=y_1\cdot p$, $w\in W$, $p\in P$.
Similarly using Lemma \ref{dual2} and Lemma \ref{cf1} it is enough to consider $(\cff, a',b')$
where $a'=w$ and $b'=x_2'\cdot p$. On the other hand, for a fixed pair $(w,p)$ and  fixed $U,V$ decomposition
the coefficient of $F_{\cf(2)}(F_{\cf(1)}(a))$ at $b$ is equal to
the coefficient of $F_{\cf(2)'}(F_{\cf(1)'}(a'))$ at $b'$.

\end{proof}

\begin{lemma}\label{prod}
Suppose that $a\in V_0(\cff)$ is the product of all the active starting circles. If the number of
active starting cicles of $\cff$ is greater than 1, 
then Theorem \ref{ki} holds for $(\cff, a)$
\end{lemma}

\begin{proof}
It is enough to consider the active part of $\cff$. 
Let $W\subset \{ 1,...,n\}$ denote the index set of the $\gamma$ arcs that connect  $x_1$ to  $x_i$ with $i\geq 2$.  Since $\cff$ is connected
it follows that $W$ is nonempty. 
Orient all these arcs away from $x_1$. Suppose $F_{\cf(2)}(F_{\cf(1)}(a))\neq 0$, where $a$ is the product of all the starting circles.
Since $F_{\cf(1)}(a)\neq 0$, it follows that either $\cf(1)$ is a split edge or it is Type $B$ or $D$. It follows from the chosen decoration
that $W$ is disjoint from the index set $U$ of $\cf(1)$. 
 Now according to the duality rule $F_{\cf(1)}(a)$ is the product of the starting circles of $\cf(2)$.
However the chosen decoration of the $\gamma _j$ arcs with $j\in W$ imply that $\cf(2)$ is not a split edge, neither Type $B$ or Type $D$, so in fact
$F_{\cf(2)}(F_{\cf(1)}(a))=0$.
\end{proof}

\begin{figure}
\mbox{\vbox{\epsfbox{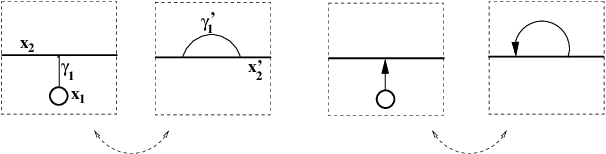}}}
\caption{ Trading}
\label{trading}
\end{figure}

\begin{figure}
\mbox{\vbox{\epsfbox{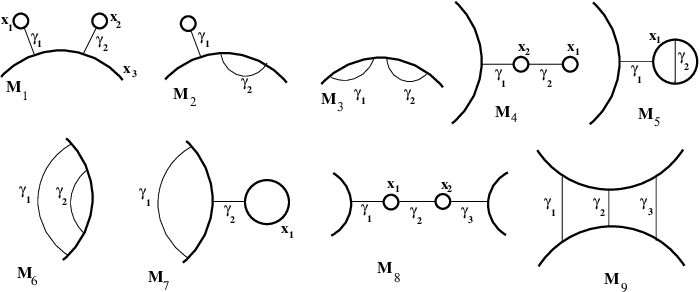}}}
\caption{}
\label{wr1}
\end{figure}

\begin{lemma}\label{pro1}
Suppose that the pair $\cff$ contains at least one of the local
 configurations $M_1,M_2,...,M_9$  in 
Figure \ref{wr1}. Then $\cff $ satisfies Theorem \ref{ki}.
\end{lemma}

\begin{figure}
\mbox{\vbox{\epsfbox{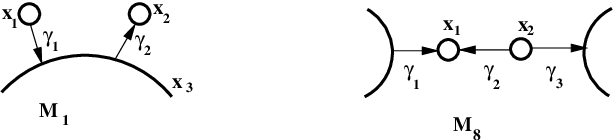}}}
\caption{}
\label{wr2}
\end{figure}

\begin{proof}
By the trading and rotation operations we can reduce the
$M_2$ and $M_3$ cases to $M _1$. Similarly $M_5$ 
can be traded to $M _4$. 
Furthermore the symmetry $\cff\rightarrow m(\cff ^\ast)$ maps
 $M_6$ to $M_4$, $M_7$ to $M_5$, and $M_9$ to $M_8$. 
So, by Lemma \ref{dual2}, \ref{trad2}
it is enough to consider the local configurations $M_1$, $M_4$ and $M_8$.

For $M_1$ 
let's write $V_0(\cff)=V(x_1)\otimes V(x_2)\otimes V(x_3)\otimes W$. It follows from Lemma \ref{cf1}
 that it is enough to consider the case $a=x_1x_2\cdot w$, where $w\in W$. We use the decoration as in Figure \ref{wr2}, and fix a point $P$ in the $x_1$ circle.
 We look at decompositions $U,V$ so that
$F_{\cf(2)}(F_{\cf(1)}(a))\neq 0$. The decoration rules out $\{1,2\}\subset U$, and $\{1,2\}\subset V$. In case $1\in U$ and $2\in V$, the filtration rule implies that
$F_{\cf(1)}(a)$ is divisible by $x_2\cdot y(P)$, and so $F_{\cf(2)}(F_{\cf(1)}(a))=0$. The case of $1\in V$, $2\in U$ is ruled out similarly.

For $M _4$ there are two cases to consider: If $a$ is divisible by $x_2$ then  the second part of Lemma \ref{cf1} applies.
If $a$ is not divisible by $x_2$ then the first part of Lemma \ref{cf2} finishes the argument. 

For $M_8      $ use the decoration in  Figure \ref{wr2}. By Lemma \ref{cf2} it is enough to consider the case when $a$ is 
divisible by $x_1x_2$. In the cases $\{1,2\}\subset U$, $\{1,3\}\subset U$ or $\{2,3\} \subset U$, the chosen decoration implies $F_{\cf(1)}(a)=0$. If 
$2\in U$, $\{1,3\}\subset V$ then $F_{\cf(1)}(a)=0$ since $a$ is divisible by $x_1x_2$. 
If $1\in U$, $\{2,3\}\subset V$, then $F_{\cf(1)}(a)$ is divisible by 
$x_2$ and so $F_{\cf(2)}(F_{\cf(1)}(a))=0$ because of the decoration.
The case of $3\in U$, $\{1,2\}\subset V$ is ruled out similarly. Finally if
 $\{1,2,3\}\subset V$ then $F_{\cf(2)}=0$ because of the decoration.

\end{proof}

\begin{figure}
\mbox{\vbox{\epsfbox{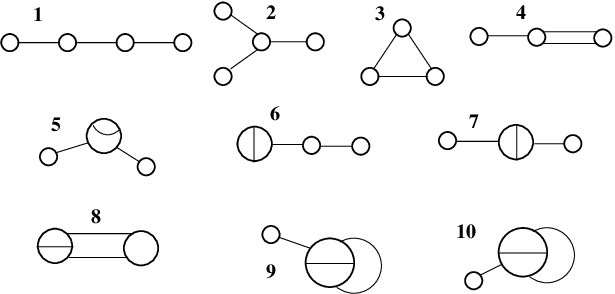}}}
\caption{ $3$-dimensional undecorated configurations
 modulo rotation, and the additional symmetry 
$\cff \longrightarrow m(\cff ^\ast)$}
\label{3D}
\end{figure}

\begin{lemma}\label{ki1}
If $\cff$ is $3$-dimensional, then Theorem \ref{ki} holds for $\cff$.
\end{lemma}

\begin{proof}
Figure \ref{3D} lists all the connected active $3$ dimensional
 configurations, modulo rotation moves, and the additional move
${\cff}\rightarrow m(\cff ^*)$. According to Lemma \ref{rot1} and \ref{dual2} it is enough to consider these cases. 
Among these, Lemma \ref{pro1}
 settles Cases 1, 2, 3, 5 and 6 and 7.
The rest of the configurations 
are given in Figure \ref{3Darrows} with a choice of decoration.

\begin{figure}
\mbox{\vbox{\epsfbox{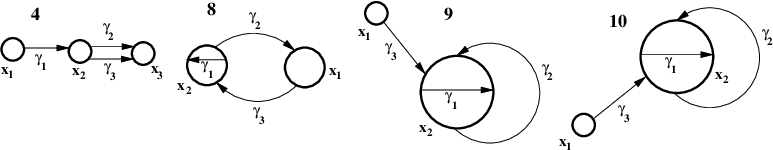}}}
\caption{}
\label{3Darrows}
\end{figure}

For Case 4 it is enough to consider $a=x_1x_3$ by Lemma \ref{cf1} and \ref{cf2}.
Using the decoration in Figure \ref{3Darrows} we see that for all 
six decompositions
$F_{\cf(2)}(F_{\cf(1)}(a))=0$.

For Case 8 it is enough to consider $a=x_1$ by Lemma \ref{cf2}. Again using the 
decoration in Figure \ref{3Darrows} we see that for all decompositions
$F_{\cf(2)}(F_{\cf(1)}(a))=0$.

For Case 9 it is enough to check $a=x_1$ by Lemma \ref{cf1}. Using the given decoration we 
have nontrivial terms from $U=\{1\}$ where $\cf(1)$ is split,
$\cf(2)$ is  type $E_{2,0}$,  and $U'=\{2,3\}$ where $\cf(1)'$ is  type
$E_{1,1}$, $\cf(2)'$ is join. 
We have 
$$F_{\cf(2)}(F_{\cf(1)}(x_1))=F_{\cf(2)'}(F_{\cf(1)'}(x_1))=y_1$$
so the terms cancel.

For Case 10 it is again enough to check $a=x_1$ by Lemma \ref{cf1}. Using the given decoration we 
have nontrivial terms from $U=\{1,3\}$ where $\cf(1)$ is type $E_{1,1}$,
$\cf(2)$ is join,  and $U'=\{2,3\}$ where $\cf(1)'$ is  type
$E_{1,1}$, $\cf(2)'$ is join. We again have 
$$F_{\cf(2)}(F_{\cf(1)}(x_1))=F_{\cf(2)'}(F_{\cf(1)'}(x_1))=y_1$$
so the terms cancel.
\end{proof}

\begin{lemma}\label{ki2}
For a fixed $(\cff,a,b)$ suppose that there exists a decoration $\pp$ on $\cff$
so that the coefficient of  $$\D_{1}(\pp)( \D_{k-1}(\pp)(a))$$ at $b$ is non-trivial.
Then Theorem \ref{ki} holds for $(\cff,a,b)$. 
\end{lemma}

\begin{proof} First we use $\D_{1}(\pp)( \D_{k-1}(\pp)(a))\neq 0$ to get some kind of
structure result for $\cff$. In this family it is enough to consider
$\cff$  that doesn't contain $M_i$ in Figure \ref{wr2}, according to Lemma \ref{pro1}.
Using  Lemma \ref{dual2}-\ref{prod} 
further reduces the problem to a finite
list of undecorated $\cff$ configurations.  Then for the remaining cases
 it will be helpful to use different
$\pp'$ decorations on  $\cff$  to simplify the calculations. These steps are spelled out below.

According to Lemma \ref{ki1} it is enough to consider the case 
when $k\geq 4$. We can also assume that
$\cff$ is active and connected.  Let $\cf=(\cff,\pp)$. Clearly there is a decomposition $\cf=\cf(1)\ast \cf(2)$ and a monomial
$z\in V_1(\cf(1))=V_0(\cf(2))$ 
so that $\cf(1)$ is $k-1$ dimensional, and the coefficients of $F_{\cf(1)}(a)$ at $z$,
$F_{\cf(2)}(z)$ at $b$ are both nontrivial. 
Now the  active part of $\cf(1)$ is a $k-1$ dimensional configuration of
Type $A_{k-1}$, $B_{k-1}$, $C_{p,q}$, $D_{p,q}$ or $E_{p,q}$, where $p+q=k-1$.
The one dimensional configuration $\cf(2)$ is determined by an additional arc
 called $\delta$. 
For $\delta$ there are $3$ cases to consider
\begin{itemize}
\item $(i)$ $\delta$ joins two active ending circles of $\cf(1)$.
\item $(ii)$ $\delta$ joins an active ending circle of $\cf(1)$ to a new $w$ circle.
\item $(iii)$ $\delta$ splits one of the active ending circles of $\cf(1)$.
\end{itemize}

If $\cf (1)$ is of Type $A_{k-1}$:\\
For case $(i)$ $\cf(1)$ and $\cf$ have the same active starting circles. 
It follows that $a=1$, $z=1$, $b=1$.
 Since  $b ^\ast$ is the product of the active circles 
 we can use Lemma \ref{dual2} and  \ref{prod}. 
In case of $(ii)$ or $(iii)$, $\delta$ meets only one of the 
$k-1$  ending circles of $\cf(1)$, so $\cff$ contains 
    the $M_9$      
 configuration.

If $\cf(1)$ is of Type $B_{k-1}$:\\
Case $(ii)$  is covered by the second part of  Lemma \ref{cf1}. Cases $(i)$ and  $(iii)$
 follow from
Lemma \ref{prod}.

\begin{figure}
\mbox{\vbox{\epsfbox{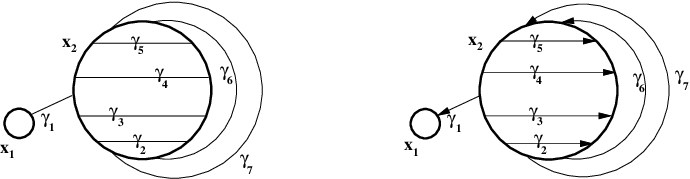}}}
\caption{}
\label{C}
\end{figure}

If  $\cf(1)$ is of Type $C_{p,q}$:\\
For case $(i)$ we only have to check for $a=1$, $b=1$. Since $b^\ast$ is the 
product of the ending circles of $\cff$, and $\cff$ has 
$p+q-2=k-1\geq 2$ ending circles,  
Lemma \ref{dual2} together with \ref{prod} implies that $(\cff,a,b)$ satisfies
Theorem \ref{ki}.

For case $(ii)$ and
 $p\geq 5$ or $p,q\geq 3$, the $\cff$ configuration contains $M_9$.
A  maximal example that we still has to consider is given by the left of Figure \ref{C}.
Modulo rotation and trading
 all the configurations without $M_9$ part can be obtained from this picture 
by deleting some  of
the $\gamma$ curves.  In particular these examples 
correspond to the index set $W\subset 
\{1,2,3,4,5,6,7\}$ with $|W|\geq 4$,
$1\in W$, $W\cap \{2,3,4,5\} \neq \emptyset$, $W\cap \{6,7\}\neq \emptyset$.  For such a $\cff (W)$
configuration use the
$\pp '$
 decoration inherited from the right side of Figure \ref{C}. 
According to Lemma \ref{cf1}
it is enough to consider $a=x_1$. Now $F_{\cf(1)}(x_1)\neq 0$ implies
that $|U|=1$ or $U=\{1,2\}$ or $U=\{1,3\}$, and in each of these  cases
$|W|\geq 4$ implies that $F_{\cf(2)}=0$.
It follows 
that
$F_{\cf(2)}(F_{\cf(1)}(x_1))=0$ for all decompositions of $(\cf(W),\pp')$.

In  Case $(iii)$ if $\delta$ is parallel to the other
$\gamma$ circles then $\cff$ is an undecorated 
$C_{p+1,q}$ and contains $M_9$. Finally if 
 $\delta$ is not parallel we can use the trading 
operation of Lemma \ref{trad2} to reduce to case $(ii)$.

If $\cf(1)$ is of Type $D_{p,q}$:\\
Case $(ii)$ follows from the second part of Lemma \ref{cf1}. Cases $(i)$ and $(iii)$ follow from Lemma \ref{prod}.


\begin{figure}
\mbox{\vbox{\epsfbox{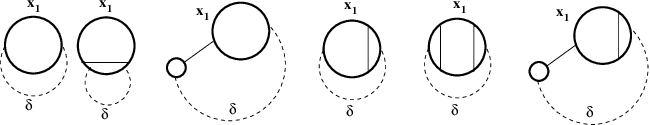}}}
\caption{ The core of the  $E\cup \delta$ configurations.}
\label{E1}
\end{figure}

\begin{figure}
\mbox{\vbox{\epsfbox{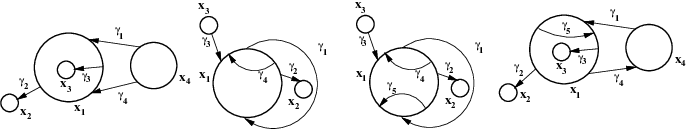}}}
\caption{ Possible $E$ configurations.}
\label{E}
\end{figure}

If $\cf(1)$ is of Type $E_{p,q}$:\\
The central starting circle of ${\cf(1)}$ gives a starting circle for $\cff$ that we denote by $x_1$.
If $\delta$ connects $x_i$ to $w$,
$x_i$ to $x_i$ or $x_i$ to $x_j$  with $i,j\geq 2$ then up to rotation
we get $M_4$, $M_5$ or $M_8$ respectively. 

In the remaining cases we argue as follows:
Since $\cf(1)$ is of Type $E$, it contains $k-1$ local configurations
 that could be traded as in  Figure \ref{trading}. Of course $\delta$ could 
intersect some of them. In fact the  endpoints of $\delta$ 
could intersect  in a  degree one  $x_i$ circle, or a 
 degree one dual $y_i$ circle.
 By deleting all
the local
configurations that are disjoint from $\delta$
 we get the core of $\cf$ see  Figure \ref{E1}. 
Each picture  contains at most two segments of $x_1$ where
the deleted configurations might have been  attached.  
However two such (disjoint from $\delta$) configurations on the same segment
  would form a local configuration of type $M_1$, $M_2$ or $M_3$ 
(up to  rotations). The first $2$ cores give at most $3$-dimensional examples.
Using rotation and trading to simplify the list,
we represent the examples
from the other $4$ cores in   Figure \ref{E}. By forgetting the decoration
for a moment, each picture on the right side of Figure \ref{E} shows three $\cff$
 configurations: the maximal $5$-dimensional, and the two 4-dimensionals
obtained by deleting $(\gamma _2,x_2)$ or $(\gamma _3,x_3)$.
Using Lemma \ref{cf1} it is enough to consider the case where
$a$ is the dual of $x_1$. Now each of these eight configurations
have  a preferred $\pp'$ decoration induced from
Figure \ref{E}. It is easy to check that for the given decorations 
 $F_{\cf(2)}(F_{\cf(1)}(a))=0$ for all decompositions.
\end{proof}

\begin{lemma}\label{ki3}
For a fixed $(\cff,a,b)$
suppose that there exists a decoration $\pp$ on $\cff$
so that the coefficient of  $$\D_{k-1}(\pp)(\D_1 (\pp)(a))$$ at $b$ is nontrivial.
Then Theorem \ref{ki} holds for $(\cff,a,b)$.
\end{lemma}

\begin{proof}
This follows from Lemma \ref{ki2} and Lemma \ref{dual2}.
\end{proof}


\begin{lemma}\label{ki4}
Suppose that there exists a decoration $\pp$ on $\cff$
so that the coefficicent of $$\D_{i}(\pp)(\D_j (\pp)(a))$$
at $b$ is non-trivial,
where $i+j=k$, $i\geq 2$, and $j\geq 2$.
Then Theorem \ref{ki} holds for $(\cff,a,b)$.
\end{lemma}

\begin{figure}
\mbox{\vbox{\epsfbox{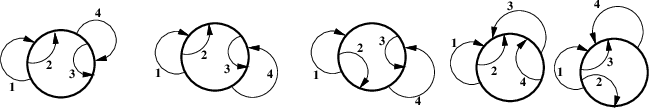}}}
\caption{}
\label{N}
\end{figure}

\begin{figure}
\mbox{\vbox{\epsfbox{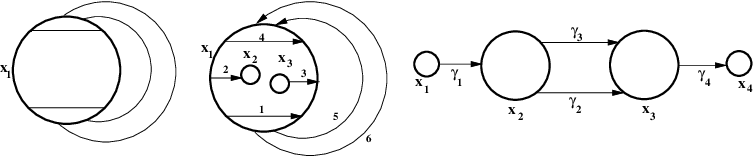}}}
\caption{}
\label{EA}
\end{figure}

\begin{proof}
If $\D_{i}(\pp) (\D_j (\pp)(a))$
at $b$ is non-trivial, then there exists a decomposition of $\cf=(\cff, \pp)$ as $\cf=\cf(1)\ast \cf(2)$, and a monomial $z\in V_1(\cf(1))
=V_0(\cf(2))$, so that
the coefficient of $F_{\cf(1)}(a)$ at $z$ is non-trivial, and $F_{\cf(2)}(z)$ at $b$ is non-trivial. By recording the types of $\cf(1)$ and $\cf(2)$ we get an 
element $Q$ in
$\{A,B,C,D,E\} ^2$:
$$Q=({\rm Type\ of} \ \cf(1), {\rm Type\ of}\ \cf(2) ).$$ 
In the special case when $\cf(1)$ or $\cf(2)$ is the $2$-dimensional configuration Type 8, we
use the extra information on $p$ and the the remark at the end of Section
\ref{higher}, to assign the value $C$ or $D$.

It follows from the definition of the $F$ map that $Q\in \{A,C\}\times \{B,D\}$ and $Q\in \{B,D\}\times \{A,C\}$ 
is not possible unless $\cf$
 is disconnected (and disconnected configuations are already discussed.)

In  case  $Q\in \{B,D\} \times \{B,D\}$
 the monomial $a$ is the product of the starting circles. According to Lemma \ref{prod} it is enough to consider the cases where $\cff$ has only one starting circle. This still leaves the possibility where both $\cf(1)$ and $\cf(2)$ is of Type  $D_{1,1}$.
 This gives  $5$ undecorated,  $\cff$ configurations. By assigning  convenient $\pp '$
decorations, we get the list of $\cf '=(\cff, \pp ')$ in
Figure \ref{N}. In all the five configurations   $F_{\cf'(1)}\neq 0$ implies $|U|=1$. Furthermore for  these choices of $U$ we have
$F_{\cf'(2)}=0$.


\begin{figure}
\mbox{\vbox{\epsfbox{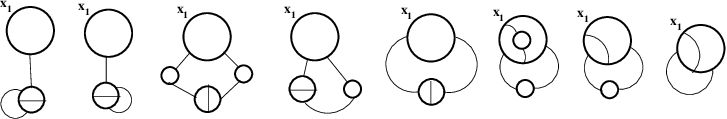}}}
\caption{}
\label{EDcores}
\end{figure}

 The case $Q\in \{A,C\}\times \{A,C\}$ follows from the previous $Q\in \{B,D\} \times \{B,D\}$ case and
Lemma \ref{dual2}.

The remaining cases all contain an $E$ factor. Similarly to the proof of case  $E$ in  Lemma \ref{ki2} we determine the
core of $\cff$, and  use  the trading and rotation operations to shorten the list of configurations to check.
(It follows from the proof of Lemma
\ref{trad2} that rotation or trading doesnt' change $Q$.)

{\bf The case  Q=(E,A) :} Let $y_1,...,y_s$ denote the active ending circles of $\cf(1)$.
Since $\cf(1)$ is Type $E$, $F_{\cf(1)}(a)$ is divisible by the central ending circle $y_1$. It follows that 
$y_1$ is a passive starting circle for $\cf(2)$. In particular
the $\gamma$ arcs of $\cf(2)$ are parallel and connect either 
$y_i$ to a new circle $w$
or $y_i$ to $y_j$, where $i,j \geq 2$. 
If the dimension of  $\cf(2)$ is greater than $2$
 we get the local configuration
 $M_8$.  

It remains to consider the case
where  $\cf(2)$ has dimension $2$. If  $w$ is an active starting circle in $\cf(2)$ then 
$w$ is a degree $2$ circle in $\cff$ and by Lemma \ref{cf2} it is enough to consider the case where $a$ is divisible
by $w$. However then $w$ divides $z$ as well, 
and since $\cf(2)$ is type $A$, we have $F_{\cf(2)}(z)=0$.
  If $y_i$ and $y_j$ are the active starting circles of $\cf(2)$ then the core of $\cff$ is given in the left of
Figure \ref{EA}. The correspoding configurations (modulo rotation and trading)
are represented in the center of Figure \ref{EA}. 
The picture represents four configurations, the maximal  $6$-dimensional
and the subconfigurations where either $(\gamma _2, x_2)$ or $(\gamma _3,x_3)$ or both are deleted.
By Lemma \ref{cf1} is enough to consider $a=x_1^\ast$. For the given decoration 
$F_{\cf'(1)}(a)\neq 0$ implies that $|U|=1$ or $U=\{3,4\}$ or $U=\{1,2\}$. With these choices of $U$ we have
$F_{\cf'(2)}=0$.

{\bf The case  Q=(A,E) :} In this case $\cf(1)$ has $s$ active ending circles $y_1,...,y_s$ where $s$ is the dimension of $\cf(1)$. 
Furthermore none of these $y_i$  divide $z$.  It follows that exactly one of these circles, say $y_1$ is the 
central starting circle of $\cf(2)$ and the others are passive starting circles for 
$\cf(2)$. So 
if $\cff$ doesn't contain $M_8$, then  $\cf(1)$ is $2$ dimensional. 
Using again the core idea, rotation and trading, we get one 
remaining $\cff$ configuration to check, see $(\cff, \pp')$  in the
right side of
Figure \ref{EA}. We need to consider $a=x_1x_4$. This gives $|U|=1$ or $U=\{2,3\}$. For these $U$ we again have
$F_{\cf(2)'}=0$.

{\bf The case  Q=(E,C):} In this case the active
starting circle of $\cf(2)$ agrees with one of the degree 
$1$ ending cirles of $\cf(1)$. It follows that if  the dimension of $\cf(1)$ is greater than $2$, then $\cff$ contains an $M_1$, $M_2$ or $M_3$ configuration
after rotations. Furthermore if $\cf(1)$ is two-dimensional, then $\cff$ 
contains an
 $M_6$ or
$M_7$ after rotations.

\begin{figure}
\mbox{\vbox{\epsfbox{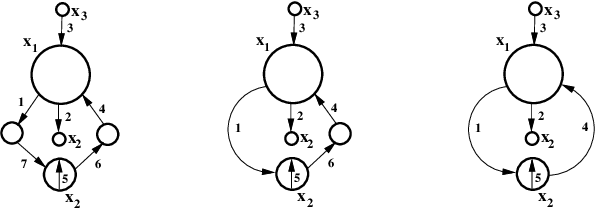}}}
\caption{}
\label{R}
\end{figure}

{\bf The case  Q=(E,D) :}
 Clearly $a$ is divisible by all the starting circles, except for $x_1$,
the central starting circle of $\cf(1)$. Since $\cf(2)$ is Type $D_{p,q}$, one of its
active starting circle is the central ending circle of $\cf(1)$, the other
$p+q-2$ active starting circles of $\cf(2)$ have to be passive ending circles
for $\cf(1)$. 
Clearly if $p$ or $q$ is large, we get an $M_8$ configuration in $\cff$.
Further simplifying with the second part of Lemma \ref{cf2} we get the  cores as in Figure \ref{EDcores}.
The first two cores give configurations that contain $M_4$ or $M_5$ 
up to rotation.  The next three cores gives
 the examples presented in 
Figure \ref{R}. The last three cores give Figure \ref{R2}. 

Note that each of the six pictures represent several $(\cff,\pp ')$
configurations. These are obtained
from the maximal configurations by erasing a
subset of the degree $1$ local configurations. (The only constraint is that
the remaining index set $W$ satisfies $|W|\geq 4$). By Lemma \ref{cf1} it is 
enough to consider $a=x_1^\ast$. 

When checking $C'(W)$ configurations in
Figure \ref{R}, $F_{\cf'(1)}(a)\neq 0$ implies that either
$|U|=1$,  $U=\{1,2\}$ or  $U=\{3,4\}$. In all these cases $F_{\cf'(2)}=0$.

When checking $C'(W)$ configurations in
Figure \ref{R2}: For the left family
 we get $|U|=1$, $U\subset  \{1,3,4,6\}$ or $U\subset \{2,5,7,8\}$. 
The middle family gives $|U|=1$, $U\subset\{3,4,6\}$ or
$U\subset \{2,5,7\}$. For the  family on the right side
we get $|U|=1$, $U=\{5,6\}$ or $U=\{3,6\}$. For all but four of these $(W,U)$ choices  
 $F_{\cf' (2)}=0$. The remaining cases correspond to  
the family on the right side of Figure \ref{R2}, with $W=\{1,2,5,6\}$ and 
$W=\{1,3,4,6\}$. For $W=\{1,2,5,6\}$  we get nontrivial 
terms with $U=\{1\}$, $U'=\{5,6\}$, both mapping $x_1 ^\ast$ to $y_1$.
For $W=\{1,3,4,6\}$ we get $U=1$, $U'=\{3,4\}$, and again both decompositions map
$x_1^\ast$ to $y_1$.

\begin{figure}
\mbox{\vbox{\epsfbox{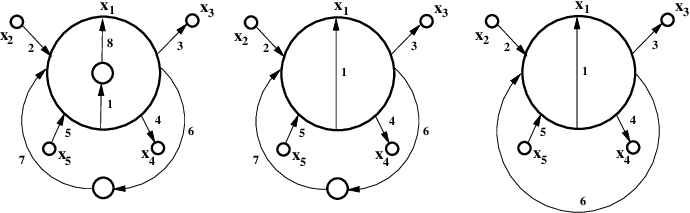}}}
\caption{}
\label{R2}
\end{figure}

The cases $Q=(B,E)$, $(E,B)$, $(E,C)$ and $(C,E)$ follow from the previous 
four cases by duality.

Finally if $Q=(E,E)$ let $r$ denote the number of circles that are both active
ending circles of $\cf(1)$ and active starting circles of $\cf(2)$. Since we are dealing with connected configurations 
we have $r\geq 1$. Let  $w$ be such a circle. If $w$ divides the monomial $z$ then  $w$ has to be the central
ending circle
of $\cf(1)$ and a degree $1$ starting circle of $\cf(2)$. If $w$ doesn't divide $z$ then $w$ has to be a degree $1$ ending
circle of $\cf(1)$ and the central starting circle of $\cf(2)$. It follows that
$r=1$ or $r=2$. When $r=1$
 $\cff$ contains at least one of the $M_1$, $M_2$, $M_3$, $M_4$ or $M_5$ configurations 
after rotations. When $r=2$ after trading we get the configurations in the middle and in the right of Figure 
\ref{R2}, that were dealt earlier.
\end{proof}

Now Theorem \ref{ki} follows immediately from Lemma \ref{ki2}, \ref{ki3} and \ref{ki4}.

\section{The spectral sequence.}

Let $L$ be an oriented link. Given a projection
$\DD$ of $L$, let $n^+(\DD)$ and $n^-(\DD)$ 
denote the number of positive and negative
crossings respectively.

Let's recall the bigrading from Khovanov homology, see \cite{Kh}, 
\cite{Ras}.
Given $I\in \{0,1\}^n$, and a monomial $z\in V(I)$, 
the homological grading is 
$$h(z)=|I|-n^-(\DD)$$
the $q$ grading
$$q(z)=gr(z)+|I|+n^+(\DD)-2n^-(\DD).$$

There is also a $\delta$ grading, see \cite{Ras2}, given by
$$\delta(z)=q(z)-2h(z)=gr(z)-|I|+n^+(\DD).$$

Note that in Khovanov boundary map $\D _1$ 
shifts the $(q,h)$ grading by $(0,1)$, and gives a bigraded homology theory.

According to the grading rule, see Lemma \ref{grading}
 the map $\D_k$ shifts the $(q,h)$ gradings 
by $(2k-2,k)$, and decreases the $\delta$ grading by $2$. 

\begin{defn}
Given a diagram $\DD $ and a decoration $\pp$, let ${\widehat C}
(\DD,\pp)$ denote the chain complex
$(C_{\DD}, \D(\pp))$, and let ${\widehat H}(\DD,\pp)$ denote its homology.
\end{defn}

Accorting to the grading shifts $\delta$ gives a well-defined grading on 
${\widehat H}(\DD, \pp)$ and we get a decomposition
$${\widehat H}(\DD, \pp)=\oplus _i {\widehat H}_i(\DD,\pp)$$

Given $i\in {\bf Z}$, let $C_i \subset {\widehat C} (\DD,\pp)$ be generated by pairs
$(I,z)$, where $z$ is a monomial in $V(I)$ with
$h(z)\geq i$.

\begin{figure}
\mbox{\vbox{\epsfbox{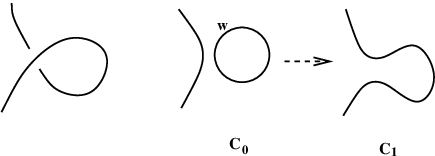}}}
\caption{}
\label{Pos}
\end{figure}

Then $C_i$ is a subcomplex of ${\widehat C} (\DD,\pp)$. The corresponding filtration by $C_i$, $i\in {\bf Z}$
induces a spectral sequence. Since the first non-trivial differential $\D(1)$ is the Khovanov differential
we get a spectral sequence starting
at the ${\Z}_2$ Khovanov homology of $L$ and ending at ${\widehat H}(\DD ,\pp)$

\begin{thm}\label{Inv}
Given an oriented link $L$, let $\DD _1$, $\DD _2$ be two diagrams representing $L$, and let $\pp_1$, $\pp _2$ be decorations for
$\DD _1$ and $\DD  _2$ respectively. Then there is a grading preserving isomorpism between 
${\widehat H}(\DD  _1,\pp _1)$ and
${\widehat H}(\DD  _2,\pp _2)$. Furthermore the two spectral sequences from
 Khovanov homology to ${\widehat H}$ are also isomorphic.
\end{thm}

\begin{figure}
\mbox{\vbox{\epsfbox{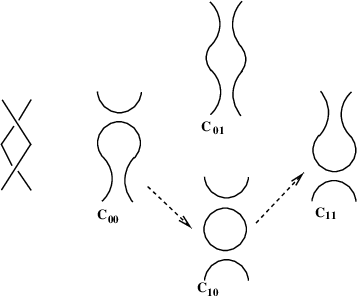}}}
\caption{}
\label{Re2}
\end{figure}

\begin{proof}
We start with a given diagram and two different decorations. Clearly it is enough to consider the case where $\pp$ and $\pp '$ differ only 
at one of the crossing, say the $m$-th crossing.  Then  we claim that 
$$G: {\widehat C} (\DD, \pp)\longrightarrow {\widehat C}(\DD,\pp ')$$
 given by $G( x)= x+H_m(x)$ is a filtered ismorphism between the chain complexes.  We have 
$$ G (\D (\pp)(x))= \D(\pp) (x)+ H_m (\D(\pp)(x))$$ and

$$\D(\pp')(x+H_m(x))=
(\D(\pp)+H_m\cdot \D(\pp) +\D(\pp)\cdot H_m )(x+ H_m((x))$$
according to Theorem \ref{Relation}. 
Since $H_m \cdot \D(\pp)\cdot  H_m=0$ and $H_m\cdot H_m=0$ we get
that $G$ is is chain map. Since $G$ is defined as the  identity plus lower order, it follows that
both the kernel and cokernel of $G$ is trivial.

The next step is to look at the Reidemeister moves. This works the same way as in Khovanov homology.
However
 for the third Reidemeister move we have to pay attention to the decorations as well.

For the first Reidemeister move with positive stabilization
we get the picture in  Figure \ref{Pos}. Then $C_0$ is a direct sum of  $C_0 ^+$ and $C_0 ^-$. Here
$C_0^-$ is generated  
by $z$ that are divisible with $w$, and $C_0^+$ is generated by $z$ that are not divisible with $w$. 
Note  that the higher differentials between $C_0^+$ and $C_1$ are trivial: Any configuration 
$\cff$ that contains the $\gamma _1$ arc will have $w$ as a degree 1 starting circle. However then
$\cff$ needs to be a Type $E$ configuration. Since $z$ is not divisible by $w$ it follows that
$F_{\cf}(z)=0$. However the Khovanov differential $\DD _1$ gives an isomorphims between $C_0^+$ and 
$C_1$. Cancelling this subcomplex with trivial homology, we  get the  $C_0 ^-$ is a quotient complex, which
is isomorphic to  the chain complex before stabilization. Since the isomorphism preserves the filtration we 
also get the same spectral sequence.

For the second Reidemeister move we again have the usual argument in Khovanov homology, see Figure \ref{Re2}.
Here we first cancel with the subcomplex spanned by $C_{10}^+ \oplus C_{11}$. In the simplified complex
we cancel with  the quotient complex $C_{00}\oplus C_{10}^-$. The remaining $C_{01}$ is isomorphic to the
original complex.

%

\begin{figure}
\mbox{\vbox{\epsfbox{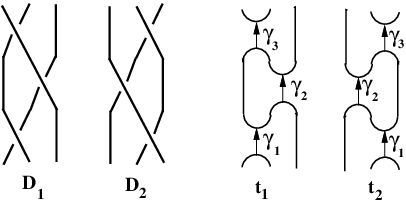}}}
\caption{}
\label{Re3b}
\end{figure}

\begin{figure}
\mbox{\vbox{\epsfbox{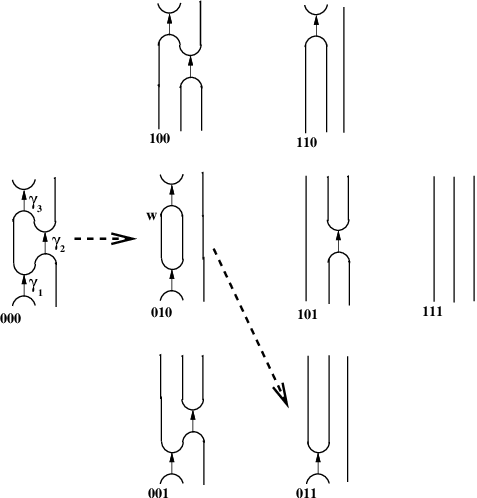}}}
\caption{}
\label{Re3}
\end{figure}

\begin{figure}
\mbox{\vbox{\epsfbox{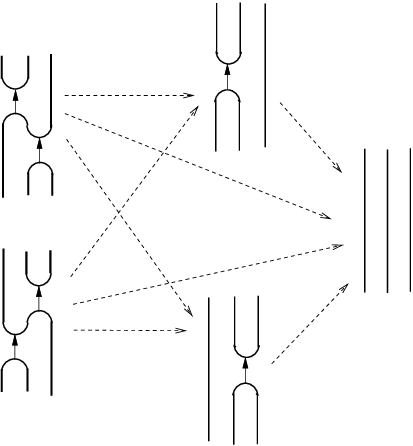}}}
\caption{}
\label{Simplified}
\end{figure}

For  the third Reidemeister move  we use the decorations as in Figure \ref{Re3b}.
Using the usual argument, see Figure \ref{Re3} we first reduce the chain complex by 
cancelling the quotient complex $C_{000} \oplus C_{010}^-$.
 Then we reduce further by contracting with the edge map that connect
$C_{010}^+ $ and $C_{011}$. 

We have to determine the resulting complex.
The first observation is that the complexes $C_{011}$ and $C_{110}$ are isomorpic. 

When
considering the restriction of the boundary map  from $C_{010} ^+$ to $C_{110}$ and 
$C_{010}^+$ to $C_{001}$ the one-dimensional configurations  induce an isomorphism. 
The higher dimensional configurations contribute trivially:  such a $\cf$ configuration
would contain $w$ as a degree one circle, so it has to be Type $E$, but $a\in C_{010}^+$
is not divisible with $w$.

Finally we claim that the map from $C_{010}^+$ to $C_{111}$ is trivial:
A configuration $\cf$ that connects them would contain $w$ as an active starting circle, 
and both $\gamma _2$ and $\gamma _3$. Because of the decoration it is enough to consider the case where.
 $\cf$ is Type $B$ or $D$. However  $a$ not divisible by $w$ implies that $F_{\cf}(a)=0$.

It follows that the simplified chain complex is identical to the chain complex in Figure \ref{Simplified}. Note that here
 the maps are still given 
by the usual decorated configurations, where we use small upwards or downwards
 isotopies supported in the region to identify the circles.  
A similar argument shows that ${\widehat C}(\DD _2,\pp _2)$ can be simplified to the same complex.

\end{proof}

\section{Further construction}

First note that there is an alternating construction $\D '$ for the boundary map, where $F'_{\cf}$ is defined by
$$F'_{\cf}= F_{m(\cf)}.$$
So for example in this alternate world the 2 dimensional Type $8$ 
configuration would contribute trivially, and for a Type $16$ 
configuration $\cf$ 
we would have $F'_{\cf}(1)=1$ and $F'_{\cf}(x_1x_2)=y_1y_2$. Clearly the construction carries through with $\D '$ in place of $\D$ and gives a twin 
version ${\widehat H}'(L)$ of ${\widehat H}(L)$.

For braids one could simplify the choice of decorations as follows. First draw the braid diagram so that the strands are moving in the 
vertical direction. Then at each $0$ resolution the $\gamma$ arc is either vertical or horizontal. Now orient the arcs so that vertical arcs are oriented upward, and the horizontal arcs are oriented to the right.

For the transverse element in ${\widehat H}(L)$ we follow the construction in \cite{Plam},
see also \cite{Baldwin}: Given an $n$-braid we choose the
unique resolution that gives back the $n$-parallel circles. For this resolution we choose the monomial $z$ that is the product of 
all the $n$ circles. Now using the braid decoration as above we see that $z$ is a closed element in ${\widehat C}(\DD, \pp)$: The 
choice of $z$ implies that we only have to check Type $B$ or Type $D$ configurations. The geometry of the resolution shows that the only possibility is Type $B_2$. However all the oriented arcs are pointed to the right, so there are no Type $B_2$ configurations that emanate from this resolution.
Now we have to study braid moves as in \cite{Plam}, and see that the distinguished element maps to each other. 
However that is straightforward from the discussion of
Section 6, where we view the second and the third moves as supported inside a braid. (Note that the braid decoration guarantees that 
both $\DD_1$ and $\DD_2$ are decorated as in Figure \ref{Re3b}.)

For a link with a distinguished component, we can define a reduced version ${\widehat H}_{red}(L)$. This is given by the usual point
 filtration in Khovanov homology. Take a point $P$ in the diagram that lies in the distinguished component. Then for each resolution 
we get a special circle $x(P)$ that contains $P$. Now define
let 
${\widehat C}(\DD,\pp, P)$ be generated by those monomials that are divisible by $x(P)$. According to the Filtration rule of Section 2,
we see that ${\widehat C}(\DD,\pp, P)$ is a subcomplex of ${\widehat C}(\DD, \pp)$. 

Since the edge homotopies  $H_m$ map the subcomplex to itself, the proof 
of invariance in  Theorem \ref{Inv} carries through for the reduced theory as well, and we get 
${\widehat H}_{red}(L)$. Note that here the $q$ and $\delta$  gradings are slightly modified, 
in particular both of them are shifted up by 1, so that  the reduced Khovanov homology
and the reduced ${\widehat H}$ of the unknot $U$ are supported in $\delta$ grading $0$.

Computations for ${\widehat H}(K)$  and ${\widehat H} _{red}(K)$ are given 
in \cite{Seed} for large families of knots. A particularly interesting case
is the $(3,5)$ torus knot for which ${\widehat H}(T_{3,5})= \Z_2\oplus \Z_2$
supported in $\delta$ gradings $1$ and $3$, and 
${\widehat H}_{red}(T_{3,5})=Z_2$
supported in $\delta$ grading $2$, compare 
also with \cite{Baldwin} and \cite{Bloom}.
The computations of Seed and the known computations for ${\widehat {HF}}$
for branched double covers, see \cite{Greene}, raise the natural question.

\begin{question}
For knots $K$ in $S^3$ is ${\widehat H}_{red}(K)$ always isomorphic to
 ${\widehat {HF}}(\Sigma _K)$ as  mod 2 graded vector spaces over $\Z_2$?
\end{question}

Another natural problem is the following
\begin{conj}
Let $U$ denote the unknot. If $K$ is a prime knot so that 
 ${\widehat H}_{red}(K)$ and 
 ${\widehat H}_{red}(U)$ are  isomorphic
as $\delta$ graded vector spaces, then
$$K=U.$$
\end{conj}

\end{document}